\documentclass[a4paper,10pt]{amsart}
\usepackage[all]{xy}
\usepackage{amsmath,amssymb,amscd,bbm,amsthm,mathrsfs,yhmath}
\usepackage[colorlinks=true,citecolor=blue,linkcolor=blue]{hyperref}
\usepackage{float}
\usepackage{threeparttable}
\usepackage{multirow}
\usepackage{amsfonts}

\usepackage[justification=centering]{caption}

\newtheorem{thm}{Theorem}[section]
\newtheorem{cor}[thm]{Corollary}
\newtheorem{lem}[thm]{Lemma}
\newtheorem{prop}[thm]{Proposition}
\newtheorem{defn}[thm]{Definition}
\newtheorem{example}[thm]{Example}

    \def\ot{\otimes}

\def\Ext{\operatorname {Ext}}

\def\Mod{\operatorname {Mod}}

\def\k{\mathbbm{k}}

\def\GL{\operatorname{GL}}
\def\id{\operatorname {id}}

\def\SGMCM{\operatorname {\underline{CM}^{\mathbbm{Z}}}}

\def\GrAut{\operatorname{GrAut}}

\def\Aut{\operatorname{Aut}}
\def\gld{\operatorname{gldim}}

\def\mod{\operatorname {mod}}

\def\sym{\operatorname {Sym}}
\def\grm{\operatorname{gr}}
\def\Gr{\operatorname{Gr}}
\def\tor{\operatorname{tor}}
\def\qgr{\operatorname{qgr}}

\title[Non-Comm. Conics asso. to Symm. RSP]{Classification of Noncommutative Conics associated to Symmetric Regular Superpotentials}

\author{Haigang Hu} 
\address{Graduate School of Science and Technology, Shizuoka University, Ohya 836, Shizuoka 422-8529, Japan
}
\email{h.hu.19@shizuoka.ac.jp}

\begin{document}

\begin{abstract}
Let $S$ be a $3$-dimensional quantum polynomial algebra, and $f \in S_2$ a central regular element. The quotient algebra $A = S/(f)$ is called a noncommutative conic. For a noncommutative conic $A$, there is a finite dimensional algebra $C(A)$ which determines the singularity of $A$. In this paper, we mainly focus on a noncommutative conic such that its quadratic dual is commutative, which is equivalent to say, $S$ is determined by a symmetric regular superpotential. We classify these noncommutative conics up to isomorphism of the pairs $(S,f)$, and calculate the algebras $C(A)$. 
\end{abstract}
 
\maketitle

\thispagestyle{empty}

\section{Introduction}

Let $S$ be an $n$-dimensional quantum polynomial algebra, and $f \in S_2$ a central regular element of $S$. The quotient algebra $A: = S/(f)$ is called a {\it noncommutative quadric hypersurface}. The classification of noncommutative quadric hypersurfaces is a big project in noncommutative algebraic geometry.

Smith and Van den Bergh introduced a finite dimensional algebra $C(A)$ associated to $A$. Let $\SGMCM A$ be the stable category of graded maximal Cohen-Macaulay modules over $A$, $\mod C(A)$ the category of finitely generated right $C(A)$-modules, and $\operatorname{D^b}(\mod C(A))$ the associated bounded derived category. Smith and Van den Bergh proved that $\SGMCM A$ is equivalent to $\operatorname{D^b}(\mod C(A))$ as triangulated categories (\cite[Proposition 5.2]{SmV}). Moreover, He and Ye proved that $A$ is a noncommutative isolated singularity if and only if $C(A)$ is semisimple (\cite[Theorem 4.6]{HY2}). It turns out that the algebra $C(A)$ is an important tool to study the noncommutative quadric hypersurface $A$ (see also \cite{MoU}). 

Let $S^!$ be the quadratic dual of $S = T(V)/(R)$. He and Ye introduced the notion of Clifford deformation of $S^!$ in  \cite{HY1}.  A Clifford deformation $S^!(\theta)$ of $S^!$ is a nonhomogeneous PBW deformation (\cite[Chapter 5]{PolPos}) defined by a certain linear map $\theta: R^\perp \to \k$. They showed that a Clifford deformation is a strongly $\mathbbm{Z}_2$-graded algebra and its degree zero part $S^!(\theta)_0 \cong C(A)$ (\cite[Propositions 3.6 and 4.3]{HY1}). In this paper, we calculate algebras $C(A)$ by Clifford deformations. 

We call $A = S / (f)$ a {\it noncommutative conic} when $S$ is a $3$-dimensional quantum polynomial algebra. By \cite{D,BoScW,MoSm1}, every $3$-dimensional quantum polynomial algebra $S$ is isomorphic to a derivation quotient algebra $\mathcal{D}(\omega)$ for some unique (up to scalar) regular twisted superpotential $\omega \in V^{\ot 3}$. We call $\omega$ {\it symmetric} if $\omega$ is a symmetric tensor. In this paper, we are interested in a noncommutative conic $A$ such that its quadratic dual $A^!$ is commutative. We find that $A^!$ is commutative if and only if $S$ is determined by a symmetric regular twisted superpotential (see Theorem \ref{thm-srs}). Itaba and Matsuno give a complete list of regular twisted superpotentials (\cite{IMa,Ma}). Thus we can easily find symmetric regular twisted superpotentials among them. It turns out that there are four types of symmetric regular twisted superpotentials: Type S, Type S$'$, Type NC, Type EC, and they are all symmetric regular superpotentials (see Lemma \ref{lem-pre-type}). 

We say two pairs $(S,f), (S',f')$ are isomorphic if there is an isomorphism of graded algebras $\varphi: S \to S'$ such that $\varphi(f) = f'$. In this paper, we classify noncommutative conics associated to symmetric regular superpotentials up to isomorphism of the pairs $(S,f)$ (see Theorem \ref{ec-gra-ale}, Theorem \ref{fphif}, and Theorem \ref{prop-tab-sf-1}), we also give the corresponding algebras $C(A)$ except for Type EC (see Theorem \ref{ca-thm1}). For some calculations of algebras $C(A)$ of Type EC, see Example \ref{ca-tpye-ec}.

This paper is organized as follows. In Section \ref{sec-pre}, we explain definitions and notations in this paper and give some propositions we will need. In Section \ref{srs}, we give some properties about symmetric regular superpotentials. In Section \ref{ca}, we give our main results about the classification of noncommutative conics associated to symmetric regular superpotentials. We also find that the set of isomorphism classes of $4$-dimensional commutative Frobenius algebras is equal to the set of isomorphism classes of algebras $C(A)$ for noncommutative conics $A$ associated to symmetric regular superpotentials (see Corollary \ref{ca-thm2}).

\section{Preliminaries} \label{sec-pre}

Throughout this paper, $\k$ is an algebraically closed field of characteristic zero, and $\k^\times$ is the subset of $\k$ consisting of nonzero elements. All algebras, vector spaces and tensor products are over $\k$. A graded algebra is a $\mathbb{Z}$-graded algebra unless otherwise mentioned. A graded algebra $A$ is said to be locally finite if $\dim_\k A_n < \infty$ for all $n \in \mathbbm{Z}$. A connected graded algebra $A$ is a graded algebra with $A_n = 0$ for all $n < 0$ and $A_0 = \k$. 

Let $V$ be a  finite dimensional vector space. We write $\GL(V)$ for the general linear group of $V$, and write $T(V)$ for the tensor algebra over $V$.

Let $A$ be a graded algebra. We write $\Gr A$ for the category of graded right $A$-modules with morphisms preserving gradings, and $\grm A$ for the full subcategory of $\Gr A$ consisting of finitely generated graded right $A$-modules. If in addition $A$ is noetherian, let $\tor A$ be the full subcategory of $\grm A$ consisting of finite dimensional graded right $A$-modules. We will call $A$ a {\it noncommutative isolated singularity} if the quotient category $\qgr A : = \grm A / \tor A$ has finite global dimension.

\subsection{Quadratic algebras and Koszul algebras} Let $V$ be a finite dimensional vector space. A {\it quadratic algebra} $A = T(V)/(R)$ is a quotient algebra of the tensor algebra $T(V)$ where $R \subset V \ot V$. The {\it quadratic dual} of $A$ is the quadratic algebra $A^!: = T(V^*)/(R^\perp)$ where $V^*$ is the dual vector space and $R^\perp \subset V^* \ot V^*$ is the orthogonal complement of $R$.  Since $R^* \cong A^!_2 $, a linear map $\theta : R \to \k$ can be regarded as an element in $A^!_2$.

A locally finite connected graded algebra $A$ is called a {\it Koszul algebra} (\cite{ShTi}) if the trivial module $\k_A$ has a free resolution
$$
\xymatrix{
\cdots \ar[r] & P^n \ar[r] & \cdots \ar[r] & P^1 \ar[r] & P^0 \ar[r] & \k \ar[r] & 0
}
$$
where $P^n$ is a graded free module generated in degree $n$ for each $n \geq 0$. A Koszul algebra $A$ is a quadratic algebra, and its quadratic dual $A^!$ is also a Koszul algebra (\cite{ShTi}).

\subsection{Quantum polynomial algebras} 
Here we follow definitions in \cite{MoU} (see also \cite{ShTi}).
\begin{defn}{\rm 
A noetherian connected graded algebra $S$ is called a $d$-dimensional {\it quantum polynomial algebra} if it satisfies
\begin{itemize}
\item[(1)] the global dimension $\gld {S} = d < \infty$,
\item[(2)] (Gorenstein condition) $\Ext^i_S(\k,S) \cong \left\{ \begin{matrix} \k & \text{if} \ i = d,  \\ 0 & \text{if} \ i \neq d, \end{matrix} \right.$ 
\item[(3)] $H_S(t) = (1-t)^{-d}$ 
\end{itemize}
where $H_S(t) : = \sum_n (\dim_{\k} S_n) t^n$ is the Hilbert series. 
}
\end{defn}

When $S$ satisfies conditions $(1),(2)$ above, $S$ is called an {\it Artin-Schelter regular algebra} (\cite{AS}). Note that a quantum polynomial algebra $S$ is Koszul (\cite[Theorem 2.2]{ShTi}). Artin, Tate and Van den Bergh proved that a $3$-dimensional quantum polynomial algebra is a domain (\cite{ATV}).

\subsection{Noncommutative quadric hypersurfaces} 

Let $S$ be a quantum polynomial algebra, and $f \in S_2$ a central regular element of $S$. The quotient algebra $A = S/(f)$ is called a noncommutative quadric hypersurface. When $S$ is a $3$-dimensional quantum polynomial algebra, we will call $A$ a {\it noncommutative conic}. Smith and Van den Bergh defined the algebra (\cite{SmV})
$$
C(A) : = A^! [(f^!)^{-1}]_0
$$ 
where $f^!$ is explained in the following proposition. 

\begin{prop} \label{nqh-prop1}
Let $S$ and $A$ be defined above. Then
\begin{itemize}
\item[(1)] $A$ is a Koszul algebra,
\item[(2)] there is a central regular element $f^! \in A^!_2$ unique up to scalar  such that $A^! / (f^!) = S^!$,
\item[(3)] $\dim_{\k}C(A) = 2^{n-1}$,
\item[(4)] two triangulated categories $\SGMCM A$ and $\operatorname{D^b} (\mod C(A))$ are equivalent,
\item[(5)] $A$ is a noncommutative isolated singularity if and only if $C(A)$ is semisimple. 
\end{itemize}
\end{prop}

\begin{proof}
$(1), (2), (3), (4)$ follow from \cite[Lemma 5.1, Proposition 5.2]{SmV} (see also \cite[Lemmas 4.11 and 4.13]{MoU}). $(5)$ follows from \cite[Theorem 4.6]{HY2}.
\end{proof}

\subsection{Clifford deformations} Let $S^! = T(V^*)/(R^\perp)$ be the quadratic dual of a Koszul Artin-Schelter regular algebra $S = T(V)/(R)$. He and Ye defined the Clifford deformation of $S^!$ (\cite[Definition 2.1]{HY1}) which is a special case of nonhomogeneous PBW deformations (\cite[Chapter 5]{PolPos}). 

\begin{defn} \label{defn-cd} {\rm
Let $\theta : R^\perp \to \k$ be a linear map. If $(\theta \ot 1- 1 \ot \theta)(V^* \ot R^\perp \cap R^\perp \ot V^*) = 0$, then we call $\theta$ a {\it Clifford map} of the Koszul algebra $S^!$, and call the associative algebra
$$
S^!(\theta) : = T(V^*) / (r - \theta(r): r \in R^\perp)
$$
the {\it Clifford deformation} of $S^!$ associated to $\theta$.
}
\end{defn}

We may view a tensor algebra $T(V)$ as a $\mathbbm{Z}_2$-graded algebra by setting $T(V)_0 = \bigoplus_{n \geq 0} V^{\ot 2n}$ and $T(V)_1 = \bigoplus_{n \geq 1} V^{\ot 2n - 1}$. Consider the Clifford deformation $S^!(\theta) = T(V^*) / (r - \theta(r): r \in R^\perp)$. Since $R^\perp \subset V^* \ot V^* \subset T(V^*)_0$, it follows that the ideal $(r - \theta(r): r \in R^\perp)$ is homogeneous. Therefore, $S^!(\theta)$ is a $\mathbbm{Z}_2$-graded algebra. A linear map $\theta: R^\perp \to \k$ is a Clifford map if and only if $\theta$ as an element in $S$ is a central element with degree two (\cite[Lemma 2.8]{HY1}). We write $\theta_f$ for the Clifford map induced by a central regular element $f \in S_2$ . 

Recall that a {\it Frobenius algebra} (\cite{L}) is a finite dimensional algebra $\mathcal{A}$ together with a nondegenerate associative bilinear pairing $(-,-): \mathcal{A} \times \mathcal{A} \to \k$. Let $A = S/(f)$ be the noncommutative quadric hypersurface. He and Ye proved the following results (\cite[Lemma 3.3, Proposition 3.6, Proposition 4.3]{HY1}). 

\begin{prop} \label{cd-prop}
Let $A = S/(f)$ and $S^!(\theta_f)$ be as defined above. Then
\begin{itemize}
\item[(1)] $S^!(\theta_f)$ is a strongly $\mathbbm{Z}_2$-graded algebra, and $S^!(\theta_f)_0 \cong C(A)$, and
\item[(2)] $S^!(\theta_f)$ is a Frobenius algebra. 
\end{itemize}
\end{prop}

\begin{example}{\rm
Let $S = \k[x,y,z]$ be the polynomial algebra. Then $x^2$ is central regular. The quadratic dual of $S$ is (we still use $x,y,z$ for $x^*,y^*,z^*$)
$$
S^! = \k\langle x,y,z \rangle / (xy + yx, xz + zx, yz + zy, x^2, y^2, z^2).
$$
Then the Clifford deformation of $S^!$ associated to $\theta_{x^2}$ is 
$$
S^! (\theta_{x^2}) =  \k\langle x,y,z \rangle / (xy + yx, xz + zx, yz + zy, x^2 - 1, y^2, z^2).
$$
Moreover, the degree zero part $S^! (\theta_{x^2})_0$ is isomorphic to $\k\langle u,v \rangle/(u^2,v^2,uv+vu)$.
}
\end{example}

Recall that an algebra $\mathcal{A}$ is called {\it quasi-Frobenius}  (\cite{L}) if $\mathcal{A}$ is a left or right self-injective, and left or right noetherian ring. Note that a Frobenius algebra $\mathcal{A}$ is quasi-Frobenius (\cite[Proposition 3.14]{L}). Then we have the following lemma. 

\begin{lem} \label{cd-prop-2}
Keep notations as above. Then $C(A)$ is quasi-Frobenius. Moreover, if $A^!$ is commutative, then $C(A)$ is Frobenius. 
\end{lem}

\begin{proof}
By Proposition \ref{cd-prop}, the associated Clifford deformation $S^!(\theta_f)$ is a Frobenius, strongly $\mathbbm{Z}_2$-graded algerba and $S^!(\theta_f)_0 \cong C(A)$. Then $S^!(\theta_f)$ is quasi-Frobenius, and we have an equivalence of abelian categories $\Gr_{\mathbbm{Z}_2} S^!(\theta_f) \cong \Mod C(A)$ where $\Mod C(A)$ is the category of right $C(A)$-modules. Since $\mathbbm{Z}_2$ is a finite group, $S^!(\theta_f)$ is quasi-Frobenius if and only if it is self-injective and noetherian in $\Gr_{\mathbbm{Z}_2} S^!(\theta_f)$ (\cite[Corollaries 2.5.2 and 5.4.3]{NvO}). It implies that $C(A)$ is quasi-Frobenius. Moreover, if $A^!$ is commutative, then its localization $C(A) = A^! [(f^!)^{-1}]_0 $ is also commutative. Thus by \cite[Corollary 16.22]{L}, $C(A)$ is Frobenius. 
\end{proof}

We can view a tensor algebra $T(V)$ as a filtered algebra with the filtration: 
$$
F_i T(V) = 0 \ \text{for} \ i < 0, \ F_iT(V) = \sum_{j = 0}^i V^{\ot j} \ \text{for} \ i \geq 0.
$$ 
Since a Clifford deformation $S^!(\theta_f) = T(V^*) / (r - \theta_f(r): r \in R^\perp)$ is a quotient algebra of $T(V^*)$, $S^!(\theta_f) $ is a filtered algebra with the induced filtration. Let $J_2 : = \{r - \theta(r) \mid r \in R^\perp \}$, and $J$ be the ideal of $T(V^*)$ generated by $J_2$, i.e. $J = (r - \theta(r): r \in R^\perp)$.

\begin{lem} \label{J}
\begin{itemize}
\item[(1)] $J \cap F_2 T(V^*)  = J_2$.
\item[(2)] $J \cap F_1 T(V^*)  = 0$.
\end{itemize}
\end{lem}

\begin{proof}
(1) follows from \cite[Theorem 2.1, Chapter 5]{PolPos}.

(2) Note $J_2 \cap F_1 T(V^*) = 0$. Then use $(1)$.
\end{proof}

\subsection{Geometric algebras and elliptic curves} 

Here we mainly follow definitions and notations in \cite{Ma,Mo}. Let $V$ be a finite dimensional vector space. For a quadratic algebra $S = T(V)/(R)$, set
$$
\mathcal{V}(R): = \{ (p,q) \in \mathbbm{P}(V^*) \times \mathbbm{P}(V^*)  \mid f(p,q) = 0 \ \text{for all} \ f \in R \}.
$$

\begin{defn}[\cite{Mo}] {\rm
A quadratic algebra $S = T(V)/(R)$ is called {\it geometric} if there is a pair $(E,\sigma)$ where $E \subset \mathbbm{P}(V^*)$ is a projective variety and $\sigma$ is a $\k$-automorphism of $E$ such that
\begin{itemize}
\item[(G1):] $\mathcal{V}(R) = \{ (p ,\sigma(p)) \in \mathbbm{P}(V^*) \times \mathbbm{P}(V^*) \mid p \in E \}$, 
\item[(G2):] $R = \{f \in V \ot V \mid f(p ,\sigma(p)) = 0 \ \text{for all}\ p \in E \}$.
\end{itemize}
If $S$ is geometric, then we write $S = \mathcal{A}(E,\sigma)$. 
}
\end{defn}

\begin{lem}\cite[Remark 4.9]{Mo}  \label{eett}
Let $S = \mathcal{A}(E, \sigma)$ and $S' = \mathcal{A}(E', \sigma')$ be two geometric algebras. Then $S, S'$ are isomorphic as graded algebras if and only if $S_1 = S'_1$ and there is an automorphism $\tau$ of $\mathbbm{P}(V^*)$ which sends $E$ isomorphically onto $E'$ such that the following diagram
$$
\xymatrix{
E \ar[r]^-{\tau} \ar[d] _-{\sigma}& E' \ar[d]^-{\sigma'} \\
E \ar[r]^-{\tau} & E'
}
$$
commutes. 
\end{lem}

We say that a geometric algebra $\mathcal{A}(E, \sigma)$ is of Type EC if $E$ is an elliptic curve in $\mathbbm{P}^2$. Every elliptic curve $E$ can be written in the Hesse form 
$
E = \mathcal{V}(x^3 + y^3 + z^3 -3 \lambda xyz)
$ where $\lambda \in \k$, and $\lambda^3 \neq 1$ (\cite[Corollary 2.18]{F}). The $j$-invariant of a Hesse form is defined as
$$
j(E) = \frac{27 \lambda^3 (\lambda^3 + 8)^3}{(\lambda^3 - 1)^3}.
$$
By \cite[Theorem IV 4.1]{Ha}, we know that the $j$-invariant classifies elliptic curves in $\mathbbm{P}^2$ up to projective equivalence. A nonsingular point $p\in E$ is called a {\it flex} of $E$ if the curve $E$ intersect the tangent at $p$ with multiplicity at least three at $p$. Note that projective equivalences of elliptic curves preserve flexes \cite[Chapter 3.8]{B}.

Let $o_E: = (1:-1:0) \in E$. Then $o_E$ is a flex of $E$ (\cite{F}), and the group structure on $(E,o_E)$ is given as follows: 
\begin{equation} \label{group-str}
p + q = \left\{ \begin{matrix} (ac\beta^2 - b^2 \alpha \gamma: bc \alpha^2 - a^2 \beta \gamma: ab \gamma^2 - c^2 \alpha \beta) &\text{or} \\
(ab\alpha^2 - c^2\beta \gamma: ac\gamma^2-b^2\alpha\beta: bc\beta^2-a^2\alpha\gamma) \end{matrix} \right.
\end{equation}
where $p = (a:b:c), q = (\alpha: \beta: \gamma) \in E$. Then $o_E$ is the zero element under this addition. For a point $p = (a:b:c) \in E$, a $\k$-automorphism $\sigma_p$ of $E$ defined as 
$$
\sigma_p: E \to E, \ q \mapsto q + p
$$
is called a {\it translation} by the point $p \in E$.  We define
\begin{equation*}
\begin{aligned}
\Aut_\k (E,o_E): &= \{\sigma \in \Aut_\k E \mid \sigma (o_E) = o_E\}, \\
E[2]:&= \{p\in E \mid p + p = o_E \}.
\end{aligned}
\end{equation*}
A point $p \in E[2]$ is called a {\it $2$-torsion point}. A point $p = (a:b:c) \in E$ such that $p \neq o_E$ is a $2$-torsion point if and only if $a=b$ (\cite[Lemma 2.16]{Ma}). Let $(1:1:\xi)$ be a $2$-torsion point on $E$. Then by the Hesse form of $E$, we have $3 \lambda \xi  = 2+ \xi^3$ which has discriminant $-108 + 108 \lambda^3$. Since $\lambda^3 \neq 1$, we have that $\xi \neq 0,-2, 1 \pm \sqrt{-3},$ and $\xi^3 \neq 1$, and $-108 + 108 \lambda^3 \neq 0$. Thus the equation $3 \lambda \xi  = 2+ \xi^3$ does not have a multiple root. By \cite[Proposition 3.1]{F}, $E[2]$ is a group of degree $4$. So we can write
$$
E[2] = \{ o_E, (1:1:\xi) \mid  3 \lambda \xi  = 2+ \xi^3\}.
$$
We will call a point $(1:1:\xi) \in E[2]\setminus{o_E}$ a {\it primitive} $2$-torsion point. For a primitive $2$-torsion point $(1:1:\xi)$, we write 
\begin{equation} \label{eq-pre-ga}
\lambda(\xi) := \frac{2+\xi^3}{3\xi}.
\end{equation}

As a special case of \cite[Theorem 3.1]{Ma}, we have the following lemma.

\begin{lem} \label{lem-pre-aa}
Let $(1:1:\xi),(1:1:\xi')$ be two primitive $2$-torsion points on $E$. Then $\mathcal{A}(E, \sigma_{(1:1:\xi)})$, $\mathcal{A}(E, \sigma_{(1:1:\xi')})$ are isomorphic as graded algebras if and only if $ \tau^l(1:1:\xi) =(1:1: \xi')$ where $\tau$ is a generator of $\Aut_{\k}(E,o_E)$ and $l \in \mathbbm{Z}_{|\tau|}$.
\end{lem}

\section{Symmetric Regular Superpotentials} \label{srs}

Let $V$ be an  $n$-dimensional vector space with a basis $\{x_1, \cdots, x_n \}$. Let $\sym(m)$ be the symmetric group of degree $m>0$, then there is a group action
$$
\sym(m) \times V^{\ot m} \to V^{\ot m}, \ (\sigma, x_{i_1} \ot \cdots \ot x_{i_m}) \mapsto x_{\sigma(i_1)} \ot \cdots \ot x_{\sigma(i_m)}.
$$
We will call $\sym^m(V) = \{ v \in V^{\ot m} \mid \sigma \cdot v = v \ \text{for all}\ \sigma \in \sym(m) \}$ the space of symmetric tensors of order $n$.

\begin{lem}\label{lem1}
Let $B = \bigwedge V$ be an exterior algebra, and let $W \subset V \ot V$ be the subspace such that $B = T(V) / (W)$, which is to say $W$ is spanned by $\{x_i x_j + x_j x_i \mid 1 \leq i,j \leq n \}$. Then $V \ot W \cap W \ot V = \sym^3(V)$.
\end{lem}

\begin{proof}
It is straightforward to check that $\sym^3(V) \subset V \ot W \cap W \ot V$. Let $w \in V \ot W \cap W \ot V$. Since $w \in W \ot V$,
$$
(12) \cdot w = (12) \cdot \left(\sum_{i,j} a_{ij}(x_i x_j + x_j x_i ) \ot v_{ij} \right) = \sum_{i,j} a_{ij}(x_j x_i + x_i x_j ) \ot v_{ij} = w
$$ 
where $a_{ij} \in \k,v_{ij}\in V$. Similarly, since $w \in V \ot W$, we have that $(23) \cdot w = w$. Since $\sym(3)$ is generated by $\{ (12),(23) \}$, $w \in \sym^3(V)$.
\end{proof}

\begin{lem}\label{lem2}
Let $B = T(V)/(R)$ be a quadratic algebra. If $B^!$ is commutative, then 
$$
V \ot R \cap R \ot V \subset \sym^3 (V).
$$
\end{lem}

\begin{proof}
Let $W$ be the subspace spanned by $\{x_i x_j + x_j x_i \mid 1 \leq i,j, \leq n \}$ . By Lemma \ref{lem1}, we have
$$
V \ot W \cap W \ot V = \sym^3(V).
$$ 
The fact that $B^! = T(V^*)/(R^\perp)$ is commutative implies that $B^!$ is a quotient algebra of the polynomial algebra $T(V^*)/(W^\perp)$. It is equivalent to say that $W^\perp \subset R^\perp$, so $R \subset W$. Thus 
$$
V \ot R \cap R \ot V \subset V \ot W \cap W \ot V = \sym^3(V).
$$
\end{proof}

Let $V$ be a vector space spanned by $\{x_1,x_2,x_3 \}$. Let $\omega \in V^{\ot 3}$. We will call $\omega$ a {\it potential}. There are unique $\omega_i \in V^{\ot 2}$ such that $\omega = \sum_{i = 1}^3 x_i \ot \omega_i$. Define $\partial_{x_i} \omega := \omega_i$ for each $x_i$. The {\it derivation-quotient} algebra of $\omega$ is defined by 
$$
\mathcal{D}(\omega) : = T(V) / (\partial_{x_1} \omega, \partial_{x_2} \omega, \partial_{x_3} \omega).
$$ 
Define a $\k$-linear map 
$$
\varphi: V^{\ot 3} \to V^{\ot 3}, \ x_{i_1} \ot x_{i_2} \ot x_{i_3} \mapsto x_{i_2} \ot x_{i_3}  \ot x_{i_1}.
$$
Following the definitions in \cite{MoSm1}, if $\varphi(\omega) = \omega$, then $\omega$ is called a {\it superpotential}; if there exists $\sigma \in \GL(V)$ such that 
$$
(\sigma \ot \id \ot \id)\varphi(\omega) = \omega,
$$
then $\omega$ is called a {\it twisted superpotential}. By \cite{D,BoScW,MoSm1}, every $3$-dimensional quantum polynomial algebra $S = T(V)/(R)$ is isomorphic to a derivation-quotient algebra $\mathcal{D}(\omega)$ for some unique (up to scalar) twisted superpotential $\omega \in V^{\ot 3}$.

\begin{defn} {\rm 
For a potential $\omega \in V^{\ot 3}$,
\begin{itemize}
\item[(1)] $\omega$ is called {\it regular} if $\mathcal{D}(\omega)$ is a quantum polynomial algebra;
\item[(2)] $\omega$ is called {\it symmetric} if $\omega \in \sym^3(V)$.
\end{itemize}
}
\end{defn}

For the rest, $S = T(V)/(R) = \mathcal{D}(\omega)$ is a $3$-dimensional quantum polynomial algebra where $V$ is a vector space spanned by $\{x,y,z\}$, and $\omega \in V^{\ot 3}$ is a regular twisted superpotential. 

\begin{prop} \label{sp-prop1}
Let $A = S/(f)$ be a noncommutative conic where $f \in S_2$ is a central regular element. If $A^!$ is commutative, then $\omega$ is symmetric.
\end{prop}

\begin{proof} 
Assume $A^!$ is commutative. Since $A = S/(f)$, there is a central regular element $f^! \in A^!_2$ such that $S^! = A^!/(f^!)$ (Proposition \ref{nqh-prop1}). It implies that $S^!$ is commutative. By Lemma \ref{lem2}, $V \ot R \cap R \ot V \subset \sym^3(V)$. Moreover, by \cite[Proposition 2.12]{MoSm2}, $\omega \in V \ot R \cap R \ot V$.
\end{proof}

For symmetric regular twisted superpotentials, we have the following lemma. 

\begin{lem} \label{lem-pre-type}
There are $4$ types of symmetric regular twisted superpotentials $\omega \in V^{\ot 3}$ up to linear change of variables. They are
\begin{equation*}
\begin{aligned}
{\rm Type \ S} &: xyz + yzx + zxy + xzy + zyx + yxz, \\
{\rm Type \ S'} &: xyz + yzx + zxy + xzy + zyx + yxz + x^3, \\
{\rm Type \ NC} &: xyz + yzx + zxy + xzy + zyx + yxz + x^3 +y^3, \\
{\rm Type \ EC}  &: xyz + yzx + zxy + xzy + zyx + yxz + \xi (x^3 + y^3 + z^3)
\end{aligned}
\end{equation*}
where $\xi \in \k^\times$ and $\xi \neq -2, 1 \pm \sqrt{-3},$ and $\xi^3 \neq 1$. Moreover, the above potentials are all symmetric regular superpotentials. 
\end{lem}

\begin{proof}
Since a complete list of regular twisted superpotentials is given in \cite[Porposition 3.1]{IMa} and \cite[Theorem 4.2]{Ma}, we only need to find the symmetric potentials among them. 
\end{proof}

We write $\omega_S,\omega_{S'},\omega_{N},\omega_{E(\xi)}$ for the superpotentials of Type S, Type S$'$, Type NC, Type EC respectively. Their derivation-quotient algebras are
\begin{equation} \label{ddn}
\begin{aligned}
\mathcal{S}: &= \mathcal{D}(\omega_S) = \k \langle x,y,z \rangle / (xy + yx, yz + zy, zx + xz), \\
\mathcal{S'}: &= \mathcal{D}(\omega_{S'}) = \k \langle x,y,z \rangle / (xy + yx, yz + zy + x^2, zx + xz),\\
\mathcal{N}: &= \mathcal{D}(\omega_{N}) = \k \langle x,y,z \rangle / (xy + yx, yz + zy + x^2, zx + xz + y^2), \\
\mathcal{E_\xi}: &= \mathcal{D}(\omega_{E(\xi)}) = \k \langle x,y,z \rangle / (xy + yx + \xi z^2, yz + zy + \xi x^2, zx + xz + \xi y^2)
\end{aligned}
\end{equation}
where $\xi \in \k^\times$ and $\xi \neq -2, 1 \pm \sqrt{-3},$ and $\xi^3 \neq 1$.

For the rest, we will simply write $\mathcal{E}$ for $\mathcal{E}_{\xi}$ if the choice of $\xi$ does not matter. Let $Z(S)_2$ be the set of central elements in $S$ with degree $2$.   

\begin{lem} \label{sp-lem1}
Let $S$ be an algebra in Equation \ref{ddn}, then $Z(S)_2 = \{ ax^2 + b y^2 + cz^2 \mid a,b,c \in \k \}$.
\end{lem}

\begin{proof}
We will give calculation for $\mathcal{N}$ and $\mathcal{E}$. The other cases are similar to $\mathcal{N}$.

(1) For $\mathcal{N} = \k \langle x,y,z \rangle / (xy + yx, yz + zy + x^2, zx + xz + y^2)$. First we claim that $x^2, y^2, z^2 \in Z(\mathcal{N})_2$. For $x^2$, we have
\begin{equation*} 
\begin{aligned}
x^2 y &= x(xy) = -x(yx) = - (xy)x = yx^2, \\
x^2 z &= x(xz) = -x(zx + y^2) = -(xz)x - xy^2 = -(-zx-y^2)x -xy^2 \\ 
&= zx^2 + (y^2x - xy^2) = zx^2.
\end{aligned}
\end{equation*}
Similarly, we have 
\begin{equation*} 
\begin{aligned}
y^2 x = xy^2, y^2 z = zy^2, z^2 x = xz^2, z^2 y = yz^2.
\end{aligned}
\end{equation*}
Thus $x^2,y^2,z^2 \in Z(\mathcal{N})_2$.

It suffices to show that if $\alpha xy + \beta yz + \gamma zx \in Z(\mathcal{N})_2$ where $\alpha,\beta, \gamma \in \k$, then $\alpha = \beta = \gamma = 0$. By the order $z > y > x$,  
$$
\{yx + xy, zy + yz + x^2, zx + xz + y^2 \}
$$ 
is a Gr\"obner basis for $\mathcal{N}$ \cite[Definition 1.4]{Le}. It implies that $\{x^iy^jz^k\mid i,j,k \geq 0 \}$ is a PBW basis for $\mathcal{N}$ (\cite[Theorem 2.3]{Le}). Thus $\{x^iy^jz^k\mid i+j+k =3 \}$ is a $\k$-basis for $\mathcal{N}_3$. Then we have following unique presentations  
\begin{equation*} 
\begin{aligned}
(\alpha xy + \beta yz + \gamma zx) x &= \alpha xyx + \beta yzx + \gamma zx^2 = -\alpha x^2y + \beta xyz - \beta y^3 + \gamma x^2 z, \\
x (\alpha xy + \beta yz + \gamma zx) &= \alpha x^2y + \beta xyz + \gamma xzx = \alpha x^2y + \beta xyz - \gamma x^2z - \gamma xy^2.
\end{aligned}
\end{equation*}
Thus if $(\alpha xy + \beta yz + \gamma zx) x = x (\alpha xy + \beta yz + \gamma zx)$, then $\alpha = \beta = \gamma = 0$. 

(2) For $\mathcal{E} = \k \langle x,y,z \rangle / (xy + yx + \xi z^2, yz + zy + \xi x^2, zx + xz + \xi y^2)$ where $\xi \in \k^\times$ and $\xi \neq -2, 1 \pm \sqrt{-3},$ and $\xi^3 \neq 1$. Fix an order $x > y > z$. Then we have following equations under this order
$$
z^2 = - \frac{1}{\xi} xy -  \frac{1}{\xi}yx, zy = - \xi x^2 -yz , zx = -xz- \xi y^2. 
$$
By solving overlaps $(z^2) y = z (zy), (z^2)x = z(zx)$, we can get 
\begin{equation*}
\begin{aligned}
- \frac{1}{\xi} xy^2 - \frac{1}{\xi}yxy &= (z^2)y = z(zy) = - \xi^2 xy^2 - \frac{1}{\xi} yxy + (\xi^2 - \frac{1}{\xi}) y^2x , \\
- \frac{1}{\xi} xyx - \frac{1}{\xi} y x^2 &= (z^2)x = z(zx) = (\xi^2 - \frac{1}{\xi})x^2y - \frac{1}{\xi} xyx - \xi^2 yx^2.
\end{aligned}
\end{equation*}
Then we have $xy^2 = y^2 x, x^2y = yx^2$. Similarly, by fixing orders $z>x>y$ and $y>z>x$, we can get relations $xz^2 = z^2x, x^2z = zx^2$ and $yz^2 = z^2y, y^2z = zy^2$ respectively. Thus $x^2, y^2, z^2 \in Z(\mathcal{E})_2$. 

Let $\alpha,\beta, \gamma \in \k$. We will show that $\alpha xy + \beta yz + \gamma zx \in Z(\mathcal{E})_2$ implies that $\alpha = \beta = \gamma = 0$. Assume $f := \alpha xy + \beta yz + \gamma zx \in Z(\mathcal{E})_2$. Let $\theta_f$ be the Clifford map of $\mathcal{E}^!$ induced by $f$. Then we have Clifford deformation (we still use $x,y,z$ for $x^*,y^*,z^*$)
$$
\mathcal{E}^!(\theta_f) = \k \langle x,y,z \rangle / (x^2 - \xi yz, y^2 - \xi xz, z^2 - \xi xy, xy - yx + \alpha, yz - zy + \beta, xz - zx + \gamma).
$$
Let $J_2 = \{x^2 - \xi yz, y^2 - \xi xz, z^2 - \xi xy, xy - yx + \alpha, yz - zy + \beta, xz - zx + \gamma \}$, and $J$ be the ideal generated by $J_2$. Fix an order $x > y > z$. Then we have following equations under this order
\begin{equation*}
\begin{aligned}
y^2 &= \xi xz, yx = xy + \alpha, zy = yz + \beta, \\
z^2 &= \xi xy, zx = xz + \gamma.
\end{aligned}
\end{equation*}
By solving the overlaps $(y^2)y = y(y^2), (z^2)z = z(z^2)$, we get
\begin{equation*}
\begin{aligned}
\xi xyz + \beta \xi x = (y^2)y = y(y^2) = \xi xyz + \alpha \xi z, \\
\xi xyz = (z^2)z = z(z^2) = \xi xyz + \beta \xi x + \gamma \xi y.
\end{aligned}
\end{equation*}
Thus we have $\beta x - \alpha z, \beta x + \gamma y \in J$. By Lemma \ref{J}, $J \cap F_1 T(V^*)$ = 0, thus $\alpha = \beta = \gamma = 0$.
\end{proof}

Every $3$-dimensional quantum polynomial algebra is a domain (\cite{ATV}), so every nonzero element is regular. 

\begin{thm}\label{thm-srs}
Let $S = \mathcal{D}(\omega)$ be a $3$-dimensional quantum polynomial algebra where $\omega$ is a regular twisted superpotential, and $f \in S_2$ a central regular element. Let $A = S / (f)$ be the noncommutative conic. Then $A^!$ is commutative if and only if $\omega$ is a symmetric regular superpotential. 
\end{thm}

\begin{proof} 
$(\Rightarrow)$ By Proposition \ref{sp-prop1} and Lemma \ref{lem-pre-type}.

$(\Leftarrow)$ By assumption $S = T(V)/(R)$ is an algebra in Equation \ref{ddn}, so it implies that $R$ is spanned by $\{xy + yx + c z^2,  yz + zy + a x^2, zx + xz + b y^2 \}$ where $a,b,c \in \k$. By Lemma \ref{sp-lem1}, central regular elements in $S_2$ are all have the form $\alpha x^2 + \beta y^2 + \gamma z^2$ where $\alpha,\beta,\gamma \in \k$, thus the subspace $R_A$ of $A$ such that $A = T(V)/(R_A)$ is spanned by $ \{xy + yx + c z^2,  yz + zy + a x^2, zx + xz + b y^2, \alpha x^2 + \beta y^2 + \gamma z^2\}$. It follows that $\{x^*y^* - y^*x^*, y^*z^* -z^*y^*, z^*x^* -x^*z^* \} \subset R_A^\perp$. Thus $A^!$ is commutative. 
\end{proof}

\section{Classification of Noncommutative Conics} \label{ca}

Let $(S,f)$ be a pair where $S$ is a $3$-dimensional quantum polynomial algebra and $f \in S_2$ a central regular element. We will say that two pairs $(S,f),(S',f')$ are isomorphic if there is an isomorphism of graded algebras $\varphi: S \to S'$ such that $\varphi(f) = f'$. In this section, we classify noncommutative conics up to isomorphism of pairs. There are two steps to classify pairs $(S,f)$:
\begin{itemize}
\item[(1)] Classify algebras $S$ up to isomorphism of graded algebras.
\item[(2)] For each fixed $S$, classify pairs $(S,f)$ up to graded algebra automorphism of $S$.
\end{itemize} 

For the first step, it suffices to consider the algebra $\mathcal{E}_{\xi}$ of Type EC in Equation \ref{ddn}. Let $\lambda(\xi)$ be defined as in Equation \ref{eq-pre-ga}. The algebra $\mathcal{E}_{\xi} = \mathcal{A}(E_{\lambda(\xi)}, \sigma_{(1:1:\xi)})$ is a geometric algebra of Type EC where 
$$
E_{\lambda(\xi)} = \mathcal{V}(x^3 + y^3 + z^3 -3 \lambda(\xi) xyz)
$$ 
is an elliptic curve and $\sigma_{(1:1:\xi)}$ is a translation by a primitive $2$-torsion point $(1:1:\xi) \in E$. 

Let $\mathcal{E}_{\xi} = \mathcal{A}(E_{\lambda(\xi)}, \sigma_{(1:1:\xi)}), \mathcal{E}_{\xi'} = \mathcal{A}(E_{\lambda(\xi')}, \sigma_{(1:1:\xi')})$ be two algebras as defined above. 

\begin{lem}
If $\mathcal{E}_{\xi} \cong \mathcal{E}_{\xi'}$ as graded algebras, then $j(E_{\lambda(\xi)}) = j(E_{\lambda(\xi')}).$
\end{lem}

\begin{proof}
Since $j$-invariant classifies elliptic curves up to projective equivalence, the result follows from Lemma \ref{eett}.
\end{proof}

\begin{lem} \label{jjll}
If $j(E_{\lambda(\xi)}) = j(E_{\lambda(\xi')})$, then there is an $\bar{\xi} \in \k$ satisfying $\lambda(\bar{\xi}) = \lambda({\xi'})$ such that $\mathcal{E}_{\xi} \cong \mathcal{E}_{\bar{\xi}}$ as graded algebras.
\end{lem}

\begin{proof}
Let $E = E_{\lambda(\xi)}, E' = E_{\lambda(\xi')}$, and $o_E = (1:-1:0) \in E, o_{E'} = (1:-1:0) \in E'$. Consider group structures on $(E,o_E),(E', o_{E'})$ as defined in Equation \ref{group-str}. Since $j(E) = j(E')$, there is a projective equivalence $\tau: E \to E'$ which maps $o_E$ to $\tau(o_E)$. Since $o_E$ is a flex of $E$, $\tau(o_E)$ is a flex of $E'$. Let $\sigma$ be a translation by the point $-\tau(o_{E}) \in E'$. By \cite[Lemma 2.13]{F} and \cite[Lemma 5.3]{Mo}, $\sigma$ can be extended to an automorphism of $\mathbbm{P}^2$. Let $\psi : = \sigma \circ\tau$. Then $\psi: E \to E'$ is a projective equivalence which maps $o_E$ to $o_{E'}$ and thus preserves the group structures (\cite[Lemma 4.9]{Ha}). Let $p := (1:1:\xi) \in E$ and $\bar{\sigma} := \psi \circ \sigma _p \circ \psi^{-1}$. Then for $q \in E'$, $\bar{\sigma}(q) = q + \psi(p)$, and 
$$
\bar{\sigma}^2 (q) = q + 2\psi(p) = q + \psi(2p) = q.
$$
Thus $\bar{\sigma}$ is a translation by a primitive $2$-torsion point $\psi(p) = (1:1:\bar{\xi}) \in E'$.  Thus we can write $\bar{\sigma} = \sigma_{(1:1:\bar{\xi})}$. Then by Lemma \ref{eett}, $\mathcal{A}(E, \sigma_{(1:1:\xi)}) \cong \mathcal{A}(E', \sigma_{(1:1:\bar{\xi})})$ as graded algebras. 
\end{proof}

Fix a $j$-invariant, say $j$, we have a set $\{\mathcal{E}_{\xi} = \mathcal{A}(E_{\lambda(\xi)}, \sigma_{(1:1:\xi)}) \mid \ j(E_{\lambda(\xi)}) = j\}$. We classify these algebras $\mathcal{E}_{\xi}$ up to isomorphism of graded algebras. 

\begin{thm} \label{ec-gra-ale}
For a fixed $j$-invariant $j$, the number of isomorphism classes of graded algebras $\mathcal{E}_{\xi} = \mathcal{A}(E_{\lambda(\xi)}, \sigma_{(1:1:\xi)})$ with $j(E_{\lambda(\xi)}) = j$  is
\begin{itemize}
\item[(i)] $3$ when $j \neq 0, 12^3$. 
\item[(ii)] $1$ when $j = 0$. 
\item[(iii)] $2$ when $j = 12^3$.  
\end{itemize}
\end{thm}

\begin{proof}
Consider two algebras $\mathcal{E}_{\xi} = \mathcal{A}(E_{\lambda(\xi)},\sigma_{(1:1:\xi)}), \mathcal{E}_{\xi'} = \mathcal{A}(E_{\lambda({\xi'})},\sigma_{(1:1:\xi')})$ with the same $j$-invariant. By Lemma \ref{jjll}, we can assume $\lambda(\xi) = \lambda(\xi') =: \lambda$ without loss of generality. So we write $E = E_{\lambda}$ for the rest of this proof. By Lemma \ref{lem-pre-aa}, two graded algebras $\mathcal{A}(E,\sigma_{(1:1:\xi)}),\mathcal{A}(E,\sigma_{(1:1:\xi')})$ are isomorphic if and only if $ \tau^l(1:1:\xi) =(1:1: \xi')$ where $\tau$ is a generator of $\Aut_{\k}(E,o_E)$ and $l \in \mathbbm{Z}_{|\tau|}$. The generators for three cases {\rm(i)} $ j \neq 0, 12^3$, {\rm(ii)} $\lambda = 0$, {\rm(iii)} $\lambda = 1+\sqrt{3}$ are already given in \cite[Lemma 2.14]{Ma}. Also note that there are $3$ primitive $2$-torsion points on $E$.

(i) If $j \neq 0, 12^3$, then the generator $\tau$ of $\Aut_{\k}(E,o_E)$ is given by 
$
\tau(a:b:c) = (b:a:c)
$
where $(a:b:c) \in E$. Thus $\mathcal{A}(E,\sigma_{(1:1:\xi)}) \cong \mathcal{A}(E,\sigma_{(1:1:\xi')})$ as graded algebras if and only if $\xi = \xi'$.

(ii) When $\lambda = 0$ and thus $j = 0$, the $3$ primitive $2$-torsion points are
$$
(1: 1: \sqrt[3]{-2}), (1:1: \sqrt[3]{-2}\varepsilon), (1:1:\sqrt[3]{-2}\varepsilon^2) 
$$ 
where $\varepsilon$ is a primitive $3$rd root of unity. The generator $\tau$ of $\Aut_{\k}(E,o_E)$ when $\lambda = 0$ is given by 
$$
\tau(a:b:c) = (b:a:\varepsilon c)
$$
where $(a:b:c) \in E$. Then $\tau(1: 1: \sqrt[3]{-2}) = (1: 1: \sqrt[3]{-2}\varepsilon), \tau^2(1: 1: \sqrt[3]{-2}) = (1: 1: \sqrt[3]{-2}\varepsilon^2).$ Thus $\mathcal{E}_\xi \cong \mathcal{E}_{\sqrt[3]{-2}}$.

(iii) When $\lambda = 1+\sqrt{3}$ and thus $j = 12^3$, the $3$ primitive $2$-torsion points are 
$$
(1: 1: 1+\sqrt{3}), (1:1: -\frac{1}{2}(1+\sqrt{3} +\sqrt{6}\sqrt[4]{3})), (1:1:-\frac{1}{2}(1+\sqrt{3} -\sqrt{6}\sqrt[4]{3})).
$$ 
The generator $\tau$ of $\Aut_{\k}(E,o_E)$ when $\lambda = 1+\sqrt{3}$ is given by 
$$
\tau(a:b:c) = (a \varepsilon^2 + b\varepsilon + c: a\varepsilon+ b\varepsilon^2 +c: a+b+c)
$$
where $(a:b:c) \in E$. Then 
\begin{equation*}
\begin{aligned}
&\tau^l (1: 1: 1+\sqrt{3}) = (1: 1: 1+\sqrt{3}) \ \text{for all} \ l \in \mathbbm{Z}_{|\tau|}, \\
&\tau(1:1: -\frac{1}{2}(1+\sqrt{3} +\sqrt{6}\sqrt[4]{3})) = (1:1:-\frac{1}{2}(1+\sqrt{3} -\sqrt{6}\sqrt[4]{3})), \\
&\tau(1:1:-\frac{1}{2}(1+\sqrt{3} -\sqrt{6}\sqrt[4]{3})) = (1:1: -\frac{1}{2}(1+\sqrt{3} +\sqrt{6}\sqrt[4]{3})).
\end{aligned}
\end{equation*}
Thus $\mathcal{E}_\xi$ is isomorphic to exactly one of $\mathcal{E}_{1+ \sqrt{3}},\mathcal{E}_{-\frac{1}{2}(1+\sqrt{3} +\sqrt{6}\sqrt[4]{3})}$.

\end{proof}

For the second step of classification of pairs, we should first consider groups of graded algebra automorphisms of algebras $\mathcal{S}, \mathcal{S}',\mathcal{N},\mathcal{E}_\xi$ in Equation \ref{ddn}.

\begin{lem} \label{srs-lem-ga} 
\begin{itemize}
\item[(1)] {\rm Type S:}
$
\GrAut \mathcal{S} 
$
is generated by 
$$
\left. \left\{ \begin{pmatrix} a & 0 & 0 \\ 0 & b & 0 \\ 0 & 0 & c \end{pmatrix}, \begin{pmatrix} 0 & 1 & 0 \\ 1 & 0 & 0 \\ 0 & 0 & 1 \end{pmatrix},\begin{pmatrix} 1 & 0 & 0 \\ 0 & 0 & 1 \\ 0 & 1 & 0 \end{pmatrix} \right| a,b,c \in \k^\times \right\}.
$$
\item[(2)]  {\rm Type S$'$:}
$
\GrAut \mathcal{S'}
$ is generated by 
$$
\left. \left\{ \begin{pmatrix} a & 0 & 0 \\ 0 & a & 0 \\ 0 & 0 & a \end{pmatrix}, \begin{pmatrix} 1 & 0 & 0 \\ 0 & b & 0 \\ 0 & 0 & b^{-1}\end{pmatrix}, \begin{pmatrix} 1 & 0 & 0 \\ 0 & 0 & 1 \\ 0 & 1 & 0 \end{pmatrix} \right| \begin{matrix} a,b \in \k^\times \end{matrix}\right\}.
$$
\end{itemize}
Let $\varepsilon$ be a primitive $3$rd root of unity. 
\begin{itemize}
\item[(3)] {\rm Type NC:}
$
\GrAut \mathcal{N}$ is generated by 
$$
\left. \left\{\begin{pmatrix} a & 0 & 0 \\ 0 & a & 0 \\ 0 & 0 & a \end{pmatrix},  \begin{pmatrix} \varepsilon^2 & 0 & 0 \\ 0 & 1 & 0 \\ 0 & 0 & \varepsilon \end{pmatrix}, \begin{pmatrix} 0 & 1 & 0 \\ 1 & 0 & 0 \\ 0 & 0 & 1 \end{pmatrix} \right| a\in \k^\times \right\}.
$$
\item[(4)] {\rm Type EC:} {\rm (i)} When $\mathcal{E}_{\xi} \cong \mathcal{E}_{1+ \sqrt{3}}$, 
$
\GrAut \mathcal{E}_{1+ \sqrt{3}} 
$
is generated by 
$$
\left. \left\{ \begin{pmatrix} a & 0 & 0 \\ 0 & a & 0 \\ 0 & 0 & a \end{pmatrix},  \begin{pmatrix} \varepsilon^2 & 0 & 0 \\ 0 & 1 & 0 \\ 0 & 0 & \varepsilon \end{pmatrix}, \begin{pmatrix} \varepsilon^2 & \varepsilon & 1 \\ \varepsilon & \varepsilon^2 & 1 \\ 1 & 1 & 1 \end{pmatrix},\begin{pmatrix} 1 & 0 & 0 \\ 0 & 0 & 1 \\ 0 & 1 & 0 \end{pmatrix} \right| a \in \k^\times \right\}.
$$
{\rm (ii)} Otherwise, 
$
\GrAut \mathcal{E}_{\xi} 
$
is generated by 
$$
\left. \left\{\begin{pmatrix} a & 0 & 0 \\ 0 & a & 0 \\ 0 & 0 & a \end{pmatrix},   \begin{pmatrix} \varepsilon^2 & 0 & 0 \\ 0 & 1 & 0 \\ 0 & 0 & \varepsilon \end{pmatrix}, \begin{pmatrix} 0 & 1 & 0 \\ 1 & 0 & 0 \\ 0 & 0 & 1 \end{pmatrix},\begin{pmatrix} 1 & 0 & 0 \\ 0 & 0 & 1 \\ 0 & 1 & 0 \end{pmatrix} \right| a \in \k^\times \right\}.
$$
\end{itemize}
\end{lem}

\begin{proof}
$(1),(2),(3)$ follow from \cite{Y}. $(4)$ follows from \cite[Proposition 4.7]{Ma}.
\end{proof}

Then we provide an efficient way to check whether two pairs are isomorphic.  

\begin{thm} \label{fphif}
Let $S$ be a $3$-dimensional quantum polynomial algebra determined by a symmetric regular superpotential. Then
$$
(S, a x^2 + b y^2 + c z^2) \cong (S, a' x^2 + b' y^2 + c' z^2)
$$ 
if and only if there is $\varphi \in \GrAut S$ such that
$$
\begin{pmatrix} a' \\ b' \\ c' \end{pmatrix} = \varphi \begin{pmatrix} a \\ b \\ c \end{pmatrix}.
$$ 
\end{thm}

\begin{proof}
We will give the proof for Type EC when $\mathcal{E}_{\xi} \cong \mathcal{E}_{1+ \sqrt{3}}$. The other cases are similar. Consider $(\mathcal{E}_{1+ \sqrt{3}}, a x^2 + b y^2 + c z^2), (\mathcal{E}_{1+ \sqrt{3}}, a' x^2 + b' y^2 + c' z^2)$ where $a,b,c,a',b',c' \in \k$. By Lemma \ref{srs-lem-ga}, there are $3$ generators for $\GrAut \mathcal{E}_{1+ \sqrt{3}}$ up to scalar. Let $\varepsilon$ be a primitive $3$rd root of unity. In the first case:   
$$
\begin{pmatrix} \varepsilon^2 & 0 & 0 \\ 0 & 1 & 0 \\ 0 & 0 & \varepsilon \end{pmatrix}: a x^2 + b y^2 + c z^2 \mapsto \varepsilon a x^2 + b y^2 + \varepsilon^2 c z^2
$$
is an isomorphism of pairs if and only if
$
\begin{pmatrix} a' \\ b' \\ c' \end{pmatrix} = \begin{pmatrix} \varepsilon & 0 & 0 \\ 0 & 1 & 0 \\ 0 & 0 & \varepsilon^2 \end{pmatrix} \begin{pmatrix} a \\ b \\ c \end{pmatrix}
$ where $ \begin{pmatrix} \varepsilon & 0 & 0 \\ 0 & 1 & 0 \\ 0 & 0 & \varepsilon^2 \end{pmatrix} \in \GrAut \mathcal{E}_{1+ \sqrt{3}}$. In the second case:
$$
\begin{pmatrix} \varepsilon^2 & \varepsilon & 1 \\ \varepsilon & \varepsilon^2 & 1 \\ 1 & 1 & 1 \end{pmatrix}: a x^2 + b y^2 + c z^2 \mapsto (1-\sqrt{3})((\varepsilon a + \varepsilon^2 b + c) x^2 + (\varepsilon^2 a + \varepsilon b + c) y^2 + (a+b+c) z^2)
$$
is an isomorphism of pairs if and only if 
$
\begin{pmatrix} a' \\ b' \\ c' \end{pmatrix} = (1-\sqrt{3})\begin{pmatrix} \varepsilon & \varepsilon^2 & 1 \\ \varepsilon^2 & \varepsilon & 1 \\ 1 & 1 & 1 \end{pmatrix} \begin{pmatrix} a \\ b \\ c \end{pmatrix}
$
where $\begin{pmatrix} \varepsilon & \varepsilon^2 & 1 \\ \varepsilon^2 & \varepsilon & 1 \\ 1 & 1 & 1 \end{pmatrix} \in \GrAut \mathcal{E}_{1+ \sqrt{3}}$. In the last case: 
$$
\begin{pmatrix} 1 & 0 & 0 \\ 0 & 0 & 1 \\ 0 & 1 & 0 \end{pmatrix}: a x^2 + b y^2 + c z^2 \mapsto a x^2 + c y^2 + b z^2
$$
is an isomorphism of pairs if and only if $\begin{pmatrix} a' \\ b' \\ c' \end{pmatrix} = \begin{pmatrix} 1 & 0 & 0 \\ 0 & 0 & 1 \\ 0 & 1 & 0 \end{pmatrix} \begin{pmatrix} a \\ b \\ c \end{pmatrix}$.  
\end{proof}

\begin{example}{\rm
(1) Let $\varepsilon$ be a primitive $3$rd root of unity. $(\mathcal{E}_{1+\sqrt{3}}, x^2) \cong  (\mathcal{E}_{1+\sqrt{3}}, \varepsilon^2x^2 + \varepsilon y^2 + z^2)$ since 
$$
\begin{pmatrix} \varepsilon^2 \\ \varepsilon \\ 1 \end{pmatrix} = \begin{pmatrix} \varepsilon^2 & \varepsilon & 1 \\ \varepsilon & \varepsilon^2 & 1 \\ 1 & 1 & 1 \end{pmatrix} \begin{pmatrix} 1 \\ 0 \\ 0 \end{pmatrix}.
$$
(2) Let $a,b,c,a',b',c' \in \k^\times$, and $abc=a'b'c'$. Then $(\mathcal{E}_{\sqrt[3]{2}}, a x^2 + b y^2 + c z^2)\cong (\mathcal{E}_{\sqrt[3]{2}}, a' x^2 + b' y^2 + c' z^2)$ if and only if
$
\begin{pmatrix} a'^3 \\ b'^3 \\ c'^3 \end{pmatrix} = \sigma \begin{pmatrix} a^3 \\ b^3 \\ c^3 \end{pmatrix}
$ 
for some $\sigma \in \sym(3)$.
}
\end{example}
   
In the following, we give the classification of pairs of Type S, Type S$'$ and Type NC. We omit the Type EC case since the statement for Type EC case is tedious, and Theorem \ref{fphif} is more efficient to check whether two pairs of Tpye EC are isomorphic. 

\begin{thm} \label{prop-tab-sf-1}
The isomorphism classes of $(S,f)$ associated to symmetric regular superpotentials except for Type EC are listed in the Table \ref{tab-sf}. Moreover, for the Table \ref{tab-sf}, let $\alpha,\alpha',\beta,\beta' \in \k^\times$. Then
\begin{itemize}
\item[(1)] $(\mathcal{S'}, \alpha x^2 + y^2 + z^2) \cong (\mathcal{S'}, \alpha' x^2 + y^2 + z^2)$ if and only if $\alpha'^2 = \alpha^2$;
\item[(2)] $(\mathcal{N}, \alpha x^2 +  y^2) \cong (\mathcal{N}, \alpha' x^2 + y^2)$ if and only if $\alpha'^3 = \alpha^{\pm 3}$;
\item[(3)] $(\mathcal{N}, \alpha x^2 +  z^2) \cong (\mathcal{N}, \alpha' x^2 + z^2)$ if and only if $\alpha'^3 = \alpha^3$;
\item[(4)] $(\mathcal{N}, \alpha x^2 +  \beta y^2 + z^2) \cong (\mathcal{N}, \alpha' x^2 + \beta' y^2 + z^2)$ if and only if $(\alpha'^3,\beta'^3,\alpha' \beta') = (\alpha^3,\beta^3,\alpha \beta)$ or $(\beta^3,\alpha^3,\alpha \beta)$.

\end{itemize}
\begin{table}[h] 
\begin{threeparttable}
\caption{List of $(S,f)$ up to isomorphism.} \label{tab-sf}
\begin{tabular}{c|c} 
{\rm Type}  &$(S,f)$  \\ \hline \hline
{\rm S} & $(\mathcal{S}, x^2), (\mathcal{S}, x^2 + y^2), (\mathcal{S}, x^2 + y^2 + z^2)$ \\ \hline
{\multirow{2}{*}{{\rm S$'$}}} &  {\multirow{2}{*}{\begin{tabular}{c}$(\mathcal{S'}, x^2), (\mathcal{S'}, y^2), (\mathcal{S'}, x^2 + y^2), (\mathcal{S'}, y^2 + z^2),$\\$ (\mathcal{S'}, \alpha x^2 + y^2 + z^2)$\end{tabular}}}  \\ 
& \\ \hline
{\multirow{2}{*}{{\rm NC}}} & {\multirow{2}{*}{\begin{tabular}{c}$(\mathcal{N}, x^2), (\mathcal{N}, z^2), (\mathcal{N}, \alpha x^2 + y^2), (\mathcal{N}, \alpha x^2 + z^2),$ \\ $(\mathcal{N}, \alpha x^2 + \beta y^2 + z^2)$\end{tabular}}} \\ 
& \\ \hline
\end{tabular}
\begin{tablenotes}
\item[*] $\alpha,\beta \in \k^\times$. 
\end{tablenotes}
\end{threeparttable}
\end{table}
\end{thm}

\begin{proof}
We will give the calculation for $\mathcal{N}$. The other cases are similar. Let $\alpha,\beta \in \k^\times.$ It is clear that 
\begin{equation*}
\begin{aligned}
&(\mathcal{N},\alpha x^2) \cong (\mathcal{N}, x^2) \cong (\mathcal{N}, y^2) \cong (\mathcal{N}, \beta y^2), \\
&(\mathcal{N}, \alpha z^2) \cong (\mathcal{N}, z^2), \\
&(\mathcal{N}, x^2 + z^2) \cong (\mathcal{N}, y^2 + z^2).
\end{aligned}
\end{equation*} 
Since $(\mathcal{N}, \alpha x^2 + \beta y^2) \cong (\mathcal{N}, \frac{\alpha}{\beta} x^2 + y^2)$ where $\alpha, \beta \in \k^\times$, 
we can consider $(\mathcal{N}, \alpha x^2 + y^2)$ without loss of generality. Consider $(\mathcal{N}, \alpha x^2 + y^2), (\mathcal{N}, \alpha' x^2 + y^2)$ where $\alpha,\alpha' \in \k^\times$. Since the group of graded algebra automorphisms $\GrAut \mathcal{N}$ is already given in  Lemma \ref{srs-lem-ga}, we can use Theorem \ref{fphif} and obtain that $(\mathcal{N}, \alpha x^2 + y^2) \cong (\mathcal{N}, \alpha' x^2 + y^2)$ if and only if $\alpha' = \mu \alpha$ or $\alpha' = \mu \alpha^{-1}$ where $\mu$ is a $3$rd root of unity. Then we have $(2)$. Similarly, for $(\mathcal{N}, \alpha x^2 + \beta z^2)$ and $(\mathcal{N}, \alpha x^2 + \beta y^2 + \gamma z^2)$ where $\alpha,\beta,\gamma \in \k^\times$, we can consider $(\mathcal{N}, \alpha x^2 + z^2)$ and $(\mathcal{N}, \alpha x^2 + \beta y^2 + z^2)$ without loss of generality, and get $(3),(4)$. 

On the other hand, it is easy to see that 
\begin{equation*}
\begin{aligned}
(\mathcal{N}, x^2), (\mathcal{N}, z^2), (\mathcal{N}, \alpha x^2 + y^2), (\mathcal{N}, \alpha x^2 + z^2), (\mathcal{N}, \alpha x^2 + \beta y^2 + z^2)
\end{aligned}
\end{equation*}
are not isomorphic to each other. 
\end{proof}

Since the algebra $C(A)$ is important to study a noncommutative conic $A = S/(f)$, we give algebras $C(A)$ associated to symmetric regular superpotentials except for Type EC in the following. 

\begin{thm} \label{ca-thm1} 
The algebras $C(A)$ associated to symmetric regular superpotentials except for Type EC are listed in the Table \ref{tab-ca}. Moreover, we give quadratic duals $A^!$ of noncommutative conics $A$.
\end{thm}

\begin{table}[h] 
\begin{center}
\begin{threeparttable}
\caption{Algebras $C(A)$ and quadratic duals $A^!$.} \label{tab-ca} 
\begin{tabular}{cc|c|c|c} 
\multicolumn{2}{c|}{$(S,f)$} & $A^!$ & Conditions & $C(A)$ \\ \hline \hline
{\multirow{3}{*}{$\mathcal{S}$}} &$x^2$   & $\k[x,y,z]/(y^2,z^2)$ & &  $\k[u,v]/(u^2,v^2)$ \\  \cline{2-3} \cline{5-5}
&$x^2 + y^2$  & $\k[x,y,z]/(x^2 - y^2, z^2)$ & &$\k[u]/(u^2)^{\times 2}$ \\ \cline{2-3} \cline{5-5}
&$x^2 + y^2 + z^2$   & $\k[x,y,z]/(x^2 - y^2, x^2 - z^2)$ & &$\k^4$\\ \hline
{\multirow{6}{*}{$\mathcal{S'}$}}&$x^2$    & $\k[x,y,z]/(y^2,z^2)$ & &$\k[u,v]/(u^2,v^2)$\\ \cline{2-3} \cline{5-5}
&$y^2$   &$\k[x,y,z]/(x^2 - yz, z^2)$ & & $\k[u]/(u^4)$ \\ \cline{2-3} \cline{5-5}
&$x^2 + y^2$  & $\k[x,y,z]/(x^2 - yz - y^2, z^2)$ & &$\k[u]/(u^2)^{\times 2}$\\ \cline{2-5}
&{\multirow{2}{*}{$\alpha x^2 + y^2 + z^2$}}  & {\multirow{2}{*}{\begin{tabular}{c}$\k[x,y,z]/(y^2 - z^2,$ \\ $x^2 - yz - \alpha y^2)$\end{tabular}}} & $\alpha = \pm1$ &$\k[u]/(u^2) \times \k^2$ \\ \cline{4-5}
&  &   & $\alpha \neq \pm 1$ &$\k^4$ \\ \hline
{\multirow{5}{*}{$\mathcal{N}$}}
&{\multirow{2}{*}{$\alpha x^2 + y^2$}}   & {\multirow{2}{*}{\begin{tabular}{c}$\k[x,y,z]/$ \\ $(x^2  - yz - \alpha (y^2 -  xz), z^2)$\end{tabular}}} & $\alpha = 0$ & $\k[u]/(u^4)$ \\ \cline{4-5}
&  & & $\alpha \neq 0$ &$\k[u]/(u^2)^{\times 2}$\\ \cline{2-5} 
&{\multirow{3}{*}{$\alpha x^2 + \beta y^2 + z^2$}}  & {\multirow{3}{*}{\begin{tabular}{c}$\k[x,y,z]/(x^2 - yz - \alpha z^2,$ \\ $y^2 - xz - \beta z^2)$\end{tabular}}}  & $\Delta_{\alpha,\beta} \neq 0$ & $\k^4$\\  \cline{4-5}
& & & $\alpha = \frac{3}{4}\mu , \beta = \frac{3}{4}\mu^2$ & $\k[u]/(u^3)\times \k$ \\ \cline{4-5}
& & & Otherwise & $\k[u]/(u^2) \times \k^2$\\ \hline
\end{tabular}
\begin{tablenotes}
\item[*] $\alpha,\beta \in \k$, and $\mu$ is a $3$rd root of unity. 
\item[**] $\Delta_{\alpha,\beta}$ is the discriminant of Equation \ref{eq-nc}. \end{tablenotes}
\end{threeparttable}
\end{center}
\end{table}

\begin{proof}
We will give the calculation of $C(A)$ for $\mathcal{N}$. The other cases are similar. By Proposition \ref{cd-prop}, we know that for a pair $(S,f)$, the Clifford deformation $S^!(\theta_f)$ is a $\mathbbm{Z}_2$-graded algebra and $S^!(\theta_f)_0 \cong C(A)$. So we will give the corresponding Clifford deformations (we still use $x,y,z$ for $x^*,y^*,z^*$) and then calculate the degree zero part of Clifford deformations. 

For $(\mathcal{N}, \alpha x^2 + y^2)$ where $\alpha \in \k$, the associated Clifford deformation $\mathcal{N}^!(\theta_f)$ is 
$$
\k [x,y,z] / (x^2 - yz - \alpha, y^2 - xz - 1, z^2). 
$$
Its degree zero part $\mathcal{N}^!(\theta_f)_0$ is generated by $xy,xz,yz$. Let $u = xy, v = yz$, then $uv = xz, v^2 = 0, u^2 = \alpha uv + v + \alpha$. Thus the algebra $C(A)$ is isomorphic to $\k[u,v] / (u^2 - \alpha uv - v - \alpha, v^2)$.  If $\alpha = 0$, then $\k[u,v] / (u^2 - v, v^2) \cong \k[u]/(u^4)$. If $\alpha \neq 0$, then there is a complete set of primitive central idempotents in $C(A)$
$$
r := \frac{1}{2} + \frac{1}{2\sqrt{\alpha}}u - \frac{\sqrt{\alpha}}{4}v - \frac{1}{4\sqrt{\alpha^3}}uv, \ r': = \frac{1}{2} - \frac{1}{2\sqrt{\alpha}}u + \frac{\sqrt{\alpha}}{4}v + \frac{1}{4\sqrt{\alpha^3}}uv.  
$$
In $rC(A)$, note $\sqrt{\alpha}rv = ruv, \sqrt{\alpha}r+ \frac{1+\sqrt{\alpha^3}}{2\sqrt{\alpha}} rv = ru, rv^2 = 0$, thus $rC(A) \cong \k[u]/(u^2)$. Similarly, $r'C(A) \cong \k[u]/(u^2)$. It follows that $C(A)$ is isomorphic to $\k[u]/(u^2) \times \k[u]/(u^2)$.

For $(\mathcal{N}, \alpha x^2 + \beta y^2 + z^2)$ where $\alpha, \beta \in \k$, the associated Clifford deformation $\mathcal{N}^!(\theta_f)$ is 
$$
\k [x,y,z] / (x^2 - yz - \alpha, y^2 - xz - \beta, z^2 - 1).
$$
Its degree zero part $\mathcal{N}^!(\theta_f)_0$ is generated by $xy,xz,yz$. Let $u = yz$, then $u^2 - \beta = xz, u^3 - \beta u = xy, u^4 = 2\beta u^2 + u - \beta^2 + \alpha$. Thus the algebra $C(A)$ is isomorphic to $\k[u]/(u^4 - 2 \beta u^2 - u - \alpha + \beta^2)$. Thus $C(A)$ is determined by the quartic equation 
\begin{equation} \label{eq-nc}
u^4 - 2 \beta u^2 - u - \alpha + \beta^2 = 0
\end{equation}
which has discriminant
\begin{equation} \label{eq-nc-delta}
\Delta_{\alpha,\beta} = - 27 - 256 \alpha^3 + 288 \alpha \beta + 256 \alpha^2 \beta^2 - 256 \beta^3.  
\end{equation}
When $\Delta_{\alpha,\beta} \neq 0$, Equation \ref{eq-nc} has no multiple root, and thus $C(A) \cong \k^4$. Simple cases are $\alpha = \beta = 1, \alpha = \beta = 0$. Then we consider situations when Equation \ref{eq-nc} has multiple roots. We claim that Equation \ref{eq-nc} does not have a root of multiplicity $4$ or two double roots. If Equation \ref{eq-nc} has a root $a \in \k$ of multiplicity $4$. Then by $(u-a)^4 = u^4 - 2 \beta u^2 - u - \alpha + \beta^2$, we have
$$
-4a = 0, -4a^3 = -1
$$
which is a contradiction. If Equation \ref{eq-nc} has two double roots $a,b \in \k$, then by $(u-a)^2(u-b)^2 = u^4 - 2 \beta u^2 - u - \alpha + \beta^2$, we have
$$
-2(a+b) = 0, -2ab(a+b) = -1
$$
which is a contradiction. Assume Equation \ref{eq-nc} has a triple root $a \in \k$. Then by $(u-a)^3(u-b) = u^4 - 2 \beta u^2 - u - \alpha + \beta^2$ where $b \in \k$, we have that 
$$
-3a-b = 0, 3a^2 + 3ab = -2\beta, -3a^3 -3a^2b = -1, a^3b = \alpha + \beta^2.
$$
It follows that $\alpha = \frac{3}{4} \mu , \beta = \frac{3}{4} \mu^2$ where $\mu$ is a $3$rd root of unity. In this case, $C(A) \cong \k[u]/(u^3) \times \k$. The last situation is that Equation \ref{eq-nc} has only one double root and thus $C(A) \cong \k[u]/(u^2) \times \k^2$. After similar calculation as above, we find that there are infinite choices of $\alpha,\beta$ satisfying this condition. A specific case is that $\alpha = \beta = -\frac{1}{4}$. 

\end{proof}

By Table \ref{tab-ca} and Proposition \ref{nqh-prop1}, we have the following corollary immediately. 

\begin{cor}
Let $\alpha, \beta \in \k$. Then
\begin{itemize}
\item[(1)] $\mathcal{S}/(x^2 + y^2 + z^2)$ is a noncommutative isolated singularity;
\item[(2)] $\mathcal{S'}/(\alpha x^2 + y^2 + z^2)$ is a noncommutative isolated singularity if and only if $\alpha \neq \pm1$; 
\item[(3)] $\mathcal{N}/(\alpha x^2 + \beta y^2 + z^2)$ is a noncommutative isolated singularity if and only if $\Delta_{\alpha,\beta} \neq 0$.
\end{itemize}
\end{cor}

By Table \ref{tab-ca}, Lemma \ref{cd-prop-2}, Theorem \ref{thm-srs}, and \cite[Lemma 2.1]{dTdVPre}, we have the following result. 

\begin{cor} \label{ca-thm2}
The set of isomorphism classes of $4$-dimensional commutative Frobenius algebras is equal to the set of isomorphism classes of algebras $C(A)$ for noncommutative conics $A$ associated to symmetric regular superpotentials. They are
$$
\k^4, \k[u]/(u^2) \times \k^2, \k[u]/(u^2) \times \k[u]/(u^2), \k[u]/(u^3) \times \k, \k[u]/(u^4), \k[u,v]/(u^2,v^2).
$$
\end{cor}

At last, we give some calculations of $C(A)$ for Type EC. 

\begin{example} \label{ca-tpye-ec} {\rm 
For  $(\mathcal{E}, \alpha x^2 +\beta y^2 + z^2)$ where $\alpha,\beta \in \k$, its Clifford deformation (we still use $x,y,z$ for $x^*,y^*,z^*$) is 
$$
\k [x,y,z] / (x^2 - \xi yz - \alpha, y^2 - \xi xz - \beta , z^2 - \xi xy - 1). 
$$
The degree zero part is generated by $yz,xz,xy$. Let $u = yz, v = xz, t=xy$. Then we have
\begin{equation*}
\begin{aligned}
u^2 &= \frac{\alpha \xi^2}{(1-\xi^3)} u  + \frac{\xi}{(1-\xi^3)} v + \frac{\beta \xi}{(1-\xi^3)} t+ \frac{\beta}{(1-\xi^3)}, \\
v^2 &= \frac{\xi}{(1-\xi^3)} u + \frac{\beta \xi^2}{(1-\xi^3)} v + \frac{\alpha \xi}{(1-\xi^3)} t + \frac{\alpha}{(1-\xi^3)}, \\
t^2 &= \frac{\beta \xi}{(1-\xi^3)} u  +  \frac{\alpha \xi}{(1-\xi^3)} v +\frac{\xi^2}{(1-\xi^3)} t + \frac{\alpha \beta}{(1-\xi^3)}, \\
uv &= \xi t^2 + t, ut = \xi v^2 + \beta v, vt = \xi u^2 + \alpha u. 
\end{aligned}
\end{equation*}
If $\alpha^3 \neq 1$, then
\begin{equation*}
\begin{aligned}
t &= \frac{(1 - \xi^3)^2}{\xi(1-\alpha^3)}v^3 + \frac{(2\beta\xi^5 + \alpha^2 \xi^3 -2\beta\xi^2 -\alpha^2)}{\xi(1-\alpha^3)} v^2 + \frac{(\beta^2 \xi^4- \alpha \xi^3 + \alpha^2 \beta \xi^2  -\alpha )}{\xi(1-\alpha^3)} v + \frac{(-\alpha \beta\xi^2 + \alpha^3)}{\xi(1-\alpha^3)},\\
u &= \frac{(1-\xi^3)}{\xi} v^2 - \beta\xi v - \alpha t - \frac{\alpha}{\xi}.
\end{aligned}
\end{equation*}
 Moreover, 
\begin{equation} \label{eq-ec-2}
\begin{aligned}
 v^4 &+ \frac{3\beta\xi^2}{(\xi^3 -1)} v^3 + \frac{(3\beta^2 \xi^4 - \alpha\xi^3 - 2\alpha)}{(\xi^3 -1)^2} v^2 +\frac{ \xi^2 (\beta^3 \xi^4  - 2 \alpha \beta \xi^3 +(1+ \alpha^3) \xi  -\alpha \beta
  )}{(\xi^3 -1)^3} v \\
 &+ \frac{(-\alpha \beta^2\xi^4 + (\alpha^3\beta + \beta)\xi^2 -\alpha^2 )}{(\xi^3 -1)^3} = 0.
\end{aligned}
\end{equation}
Thus $C(A)$ is determined by Equation \ref{eq-ec-2}. Let $\Delta_{\xi,\alpha,\beta}$ be the discriminant of Equation \ref{eq-ec-2}. When $\beta^3 \neq 1$, since 
$$
\psi: (\mathcal{E}, \alpha x^2 + \beta y^2 + z^2) \to (\mathcal{E}, \beta x^2 + \alpha y^2 + z^2), \ x \mapsto y, y \mapsto x, z \mapsto z.
$$ 
is an isomorphism, we can get its associated discriminant $\Delta'_{\xi,\alpha,\beta}$ by interchanging $\alpha$ and $\beta$ in $\Delta_{\xi,\alpha,\beta}$. Actually, $\Delta'_{\xi,\alpha,\beta} = \Delta_{\xi,\alpha,\beta}$. Thus if $\alpha^3 \neq 1$ or $\beta^3 \neq 1$, then $C(A) \cong \k^4$ if and only if $\Delta_{\xi,\alpha,\beta} \neq 0$. A simple case for $\Delta_{\xi,\alpha,\beta} \neq 0$ is that $\alpha = \beta = 0$.

}
\end{example}

\section*{}

{\bf Acknowledgements.} I would first like to thank my supervisor Izuru Mori. This paper was completed under his guidance. Then I would like to thank Masaki Matsuno for many useful conversations. I would also like to thank Pieter Belmans, Ji-Wei He, Kenta Ueyama for their suggestions via email.

\end{document}